\documentclass[12pt]{amsart}
\usepackage{amsmath,amsthm,amssymb,todonotes,a4wide}
\usepackage{pdfsync}

\usepackage{esint}
\usepackage[all]{xy}

\usepackage{a4wide}

\theoremstyle{plain}

 \newtheorem{theorem}{Theorem}[section]
 \newtheorem{lem}[theorem]{Lemma}
 \newtheorem{prop}[theorem]{Proposition}

 \theoremstyle{definition}

 \newtheorem{defi}[theorem]{Definition}
 
\newtheorem{ex}[theorem]{Example}

 \newtheorem{rem}[theorem]{Remark}

 \numberwithin{equation}{section}
\numberwithin{table}{section}

\newcommand{\Syl}{\operatorname{Syl}}

\newcommand{{\cS}}{{\mathcal S}}
\newcommand{{\cK}}{{\mathcal K}}
\newcommand{{\cN}}{{\mathcal N}}
\newcommand{{\cW}}{{\mathcal W}}
\newcommand{\rep}{\operatorname {rep}}

\newcommand{{\calf}}{{\mathrm f}}
\newcommand{{\calg}}{{\mathfrak g}}
\newcommand{{\calK}}{{\mathcal K}}

\newcommand{\w}{\widetilde}
\newcommand{\wG}{{\widetilde G}}\newcommand{\wbG}{{\widetilde \bG}}

\newcommand{\Gu}{{\w G}}

\newcommand{\la}{\ensuremath{\lambda}}

\renewcommand{\labelenumi}{(\alph{enumi})}

\renewcommand{\theenumi}{\thetheorem{}(\alph{enumi})}

\newcommand{\CC}{\ensuremath{\mathbb{C}}}

\renewcommand{\o}{\overline}

\newcommand{\Zent}{\ensuremath{{\rm{Z}}}}
\newcommand{\Cent}{\ensuremath{{\rm{C}}}}

\newcommand{\NNN}{\ensuremath{{\mathrm{N}}}}

\newcommand{\ep}{\epsilon}
\def\restr#1|#2{\left.#1\right\rceil_{#2}}

\newcommand{\FF}{\ensuremath{\mathbb{F}}}

\newcommand{\calD}{{\mathcal D}}

\newcommand{\SL}{{\mathrm{SL}}}
\newcommand{\PSL}{{\mathrm{PSL}}}

\newcommand{\GL}{{\mathrm{GL}}}

\newcommand{\SU}{{\mathrm{SU}}}
\newcommand{\PSU}{{\mathrm{PSU}}}

\newcommand{\tA}{\mathsf A}

\renewcommand{\calD}{\mathcal D}

\newcommand{\calG}{\ensuremath{\mathcal G}}
\newcommand{\Res}{{\mathrm{Res}}}

\newcommand{\bG}{{\bf G}}

\newcommand{\al}{{\alpha}}

\makeatletter
\def \pseudosubsection#1 {\medskip\noindent  {\bf #1}   \smallskip}

\def\Spann<#1>{\Spann@h#1@}
\def\Spann@h#1|#2@{\left\langle\left.#1\vphantom{#2}\hskip.1em\right|\,\relax #2\right\rangle}
\def\Set#1{\Set@h#1@}
\def\Lset#1{\Lset@h#1@}
\def\Set@h#1|#2@{\left\{\left.#1\vphantom{#2}\hskip.1em\,\right|\,\relax #2\right\}}
\def\Bset@h#1in#2|#3@{\Set{{#1\qin#2}|{#3}}}
\def\Lset@h#1@{\left\{#1\right\}}
\def\spann<#1>{\left\langle#1\right\rangle}

\makeatother

\newcommand{\Inn}{\mathrm{Inn}}
\newcommand{\Aut}{\mathrm{Aut}}

\newcommand{\calM}{\ensuremath{\mathcal M}}
\newcommand{\calN}{\ensuremath{\mathcal N }}

\newcommand{\Z}{\operatorname Z}

\newcommand{\bC}{\mathbf C}
\newcommand{\bS}{\mathbf S}

\newcommand{\calP}{\mathcal P}
\newcommand{\calQ}{\mathcal Q}
\newcommand{\calO}{\mathcal O}
\newcommand{\Irr}{\operatorname{Irr}}

\makeatother

\global\long\def\ovG{\mathbf{G}}

\global\long\def\CC{\mathbb{C}}

\global\long\def\calD{\mathcal{D}}

\global\long\def\calK{\mathcal{K}}

\global\long\def\SL{\operatorname{SL}}
\global\long\def\CC{\mathbb{C}}

\global\long\def\FF{\mathbb{F}}

\global\long\def\NNN{\mathrm{N}}

\global\long\def\al{\alpha}

\newcommand{\lp}{{\ell '}}

\newcommand{\Ind}{{\mathrm{Ind}}}

\newcommand{\Omegau}{\w \Omega}
\newcommand{\Guchinull}{\w G_{\chi_0}}

\newcommand{\enumroman}{\renewcommand{\labelenumi}{(\roman{enumi})} \renewcommand{\theenumi}{\thetheorem(\roman{enumi})}}
\newcommand{\enumalph}{\renewcommand{\labelenumi}{(\alph{enumi})} \renewcommand{\theenumi}{\thetheorem(\alph{enumi})}}

\renewcommand{\to}{\rightarrow}
\newcommand{\with}{\text{ with }}
\renewcommand{\wbG}{{\w\ovG}}

\newcommand{\wh}{\widehat}
\newcommand{\bl}{\operatorname{bl}}
\newcommand{\Bl}{\operatorname{Bl}}

\newcommand{\oFF}{\o \FF}
\newcommand{\Cl}{\mathfrak{Cl}}
\newcommand{\und}{\text{ and }}
\newcommand{\forevery}{\text{ for every }}
\newcommand{\wM}{{\w M}}
\newcommand{\Mu}{{\w M}}
\newcommand{\tr}{\operatorname{tr}}

\newcommand{\wchi}{{\w\chi}}

\makeindex

\title[Inductive AM in type $\tA$]{On the inductive Alperin-McKay condition for simple groups of type $\tA$}

\begin{document}
\abstract{As a sequel to \cite{CabSpaeth2}, we verify the so-called inductive AM-condition introduced in \cite{Spaeth5} for simple groups of type $\tA$ and blocks with maximal defect. This is part of the program set up to verify the Alperin-McKay conjecture through its reduction to a problem on quasi-simple groups (see \cite{Spaeth_AM_red}) but also the missing direction of Brauer's height zero conjecture (see \cite{NavarroSpaeth}).}

\author{Marc Cabanes \and Britta Sp\"ath}
\address{M. Cabanes: Institut de Math\'ematiques de Jussieu, Universit\'e Paris 7, Batiment Sophie Germain,  75013 Paris, France.} \email{cabanes@math.jussieu.fr}

\address{B. Sp\"ath: Fachbereich Mathematik, TU Kaiserslautern, Postfach 3049, 67653 Kaiserslautern, Germany. }

\email{spaeth@mathematik.uni-kl.de}
\thanks{The second author has been supported by the Deutsche Forschungsgemeinschaft, SPP 1388  and  ERC Advanced Grant 291512.}}
\maketitle

\section{Introduction}
In \cite{CabSpaeth2} the authors proved that the inductive McKay condition from \cite{IMN} holds for the simple groups $S\in\{\PSL_n(q),\PSU_n(q)\}$, and all primes. Let us recall that, according to the main result of \cite{IMN}, this condition is to be satisfied by all finite quasi-simple groups to ensure that the McKay Conjecture holds for any finite group. Another inductive criterion, the so-called inductive AM-condition, was brought forth by the second author for the Alperin-McKay Conjecture (see \cite[Definition 7.2]{Spaeth_AM_red}), i.e., a version of McKay Conjecture dealing with blocks.

In the present paper, we show that this inductive AM-condition can be given a version relative to a fixed class of defect subgroups, that this relative version behaves well with regard to the criterion introduced in \cite[2.10]{Spaeth5} for finite groups of Lie type, and finally that it is satisfied by the blocks of maximal defect of quasi-simple groups of Lie type $\tA$. Our main result (see also Theorem~\ref{thm_PSL_ist_AM_gut+} below) can then be phrased as

\begin{theorem}\label{thm_PSL_ist_AM_gut}
The inductive AM-condition from Definition 7.2 of \cite{Spaeth_AM_red} holds for the simple groups $S\in\{\PSL_n(q),\PSU_n(q)\}$ with respect to their Sylow $\ell$-subgroups, where $\ell$ is any prime.
\end{theorem}

Proving that the bijection in \cite{CabSpaeth2} preserves $\ell$-blocks is relatively easy (see \cite[7.1]{CabSpaeth}) but the additional conditions required by the inductive AM-condition need a specific work. This is explained in Section \ref{sec2} below. We first give in Definition~\ref{def_rel_AM} our version of the inductive AM-condition relative to a single defect subgroup. We then show that it suffices to check it for every defect subgroup to get the inductive AM-condition. In passing we take the opportunity to give an equivalent property that may turn out easier to check than the one in \cite{Spaeth_AM_red}; in particular it uses only ordinary characters, no projective representations. In Section \ref{sec3}, we adapt the criterion introduced in \cite{Spaeth5} and already used in \cite{CabSpaeth2} by incorporating blocks to our situation.  Some of the methods make use of results from \cite{KoshitaniSpaeth}. Section \ref{sec_blocks} recalls and establishes some useful properties of groups of Lie type $\tA$. We conclude the proof of Theorem~\ref{thm_PSL_ist_AM_gut} in Section \ref{EndProof}.

{\bf Acknowledgement:} The first author
thanks the DFG Priority Program SPP 1388 and the ERC Advanced Grant 291512 for the support of his stays in Kaiserslautern. Both authors acknowledge the fine work atmosphere during the Journal of Algebra conference in Peking University.

\section{Reformulating the inductive AM-condition} \label{sec_notation_recall} \label{sec2}
In this section, after fixing some notation around characters and blocks of finite groups, we introduce an inductive condition (Definition~\ref{def_rel_AM}) relative to a class of defect groups. Proposition~\ref{prop2_2} shows that assumed for all defect groups it is equivalent to the inductive AM-condition from \cite[7.2]{Spaeth_AM_red}. We also show quickly how the condition from Definition~\ref{def_rel_AM} may lead to a reduction theorem relative to maximal defect (Proposition~\ref{prop2_6}).

\subsection{Notations}
For finite groups, their characters and blocks we use mainly the notations introduced in \cite{Isa} and \cite{Navarro}. 

Suppose a group $A$ acts on a finite set $X$. Then we denote by $A_{x}$ the stabilizer of $x\in X$ in $A$, analogously we denote by $A_{X'}$ the stabilizer of $X'\subseteq X$. 
If a group $A$ acts on a group $G$, there is a natural action of $A$ on $\Irr(G)$. We denote by $A_{P,\chi}$ the stabilizer of $\chi$ in $A_P$, where $P\leq G$ and $\chi$ is the character of some $A_P$-stable group. 

For $N\lhd G$ and $\chi\in \Irr(G)$ we denote by $\Irr(N\mid \chi)$ the set of irreducible constituents of the restricted character $\chi_N=\Res^G_N(\chi)$ while for $\psi \in \Irr(X)$ the set of irreducible constituents of the induced character $\Ind^X_N(\psi )$ is denoted by $\Irr(X\mid \psi)$. For a subset $\calN\subseteq \Irr(N)$ we set
\[ \Irr(G\mid \calN):=\bigcup_{\chi\in\calN}\Irr(G\mid \chi).\]

For $g\in G$, we denote by $\Cl_G(g)$ or simply $\Cl (g)$ the conjugacy class of $g$. We denote by $\Cl _G(g)^+$ the sum of elements of $\Cl_G(g)$ in any group ring of $G$.

 Let $\ell$ be a prime. Let $\calO$ be a complete valuation ring with $\calO/J(\calO )= \oFF_\ell$, the algebraic closure of the field $\FF_\ell$. We denote by $x\mapsto x^*$ the reduction map $\calO\to \oFF_\ell$.

Let $G$ be a finite group.
We denote by $\Bl(G)$ the set of $\ell$-blocks of $G$, that is the blocks of the group algebra $\oFF_\ell G$, and we recall the partition $\Irr (G)=\cup_{B\in\Bl (G)}\Irr (B)$. If $D$ is an $\ell$-subgroup of $G$, we denote by $\Bl (G\mid D)$ the set of $\ell$-blocks of $G$ having $D$ as defect group. Recall that $\Bl (G\mid D)$ is non empty only if $D$ is ($\ell$-) radical, that is $D={\rm O}_\ell (\NNN_G(D))$. Moreover let $\Irr (G\mid D)$ be the union of the $\Irr (B)$ for $B\in\Bl (G\mid D)$. We denote by $\Irr_0(G\mid D)$ the set of characters $\chi\in\Irr(G\mid D)$, such that $\chi (1)_\ell =|G:D|_\ell$. For $\chi\in\Irr(G)$ we write $\bl(\chi)$ for the $\ell$-block $b$ of $G$ with $\chi\in\Irr(b)$. For $B\in\Bl(G)$ we denote by \[\la_B:\Zent (\oFF_\ell G)\longrightarrow \oFF_\ell\]
the morphism associated to $B$ as in \cite[p. 48]{Navarro}. For $\chi\in\Irr(G)$ we denote $\la_\chi :=\la_{\bl(\chi)}$. Recall that if $g\in G$, the value of $\la_\chi$ at the sum of the conjugacy class of $g$ is $\bigl(|\Cl (g)|\chi (g)/\chi (1)\bigr)^*$ , the reduction mod $J(\calO )$ of the algebraic integer $|\Cl (g)|\chi (g)/\chi (1)$. 
Recall that $G_\lp$ denotes the set of $\ell$-regular elements of $G$.

 For $H\leq G$ and $b\in\Bl(H)$ we denote by $b^G$ the (Brauer-)induced block, defined as in \cite[p. 87]{Navarro}, when it exists.

\subsection{An inductive AM-condition relative to a defect subgroup}

The inductive AM-condition for a simple group and a prime $\ell$ has been introduced in Definition 7.2 of \cite{Spaeth_AM_red}. In \cite[Sect. 7.1]{CabSpaeth} a version relative to certain radical $\ell$-subgroups has been sketched.

\begin{defi}\label{def_rel_AM} Let $S$ be a finite non-abelian simple group and $\ell$ a prime. Let $G$ be the universal covering group of $S$ and $D$ be a non-central radical $\ell$-subgroup of $G$. Then we say that {\bf the inductive AM-condition holds for $S$ with respect to $D\Z(G)/\Z(G)$} if 
\begin{enumerate}
	\item there exists some $\Aut(G)_D$-stable subgroup $M$ such that $\NNN_G(D)\leq M\lneq G$
	\item there exists an $\Aut(G)_D$-equivariant bijection 
	\[ \Omega_D:\Irr_0(G\mid D) \longrightarrow \Irr_0(M\mid D),\]
	such that 
	\begin{itemize}
		\item $\Omega_D\bigl(\Irr_0(G\mid D)\cap \Irr(G\mid \nu)\bigr)\subseteq \Irr(M\mid \nu)$ for every $\nu\in\Irr(\Z(G))$,
		\item $\bl(\chi)=\bl\bigl(\Omega_D(\chi)\bigr)^G$  for every $\chi\in\Irr_0(G\mid D)$. 
	\end{itemize}
	\item \label{def_rel_AM_cohom}
	For every character $\chi\in\Irr_0(G\mid D)$ there exist a group $A$ and characters $\w\chi$ and $\w\chi'$ such that 
	\begin{enumerate}
		\item For $Z:=\ker(\chi)\cap \Z(G)$ and $\o G:=G/Z$ the group $A$ satisfies $\o G\lhd A$, $\Cent_A(\o G)=\Z(A)$ and $A/\Z(A)\cong \Aut(G)_\chi$,
		\item $\w\chi\in\Irr(A)$ is an extension of $\o \chi$, where $\o\chi\in \Irr(\o G)$ lifts to $\chi$,
		\item for $\o M:=MZ/Z$ and $\o D:=DZ/Z$ the character $\w\chi'\in\Irr\bigl(\o M\NNN_A(\o D)\bigr)$  is an extension of $\o\chi'$, where $\o\chi'\in\Irr(\o M)$ lifts to $\chi':=\Omega_D(\chi)\in\Irr(M)$,
		\item \label{def_rel_AM_cohom_iv}
		$\bl\bigl(\Res^A_J( \w\chi)\bigr)=\bl\bigl(\Res^{\o M\NNN_A(\o D)}_{\o M \NNN_J(\o D)}(\w\chi')\bigr)^J$ for every $J$ with $\o G\leq J \leq A$ ,
		\item $\Irr(\Z(A)\mid \w\chi)=\Irr(\Z(A)\mid \w\chi')$.
	\end{enumerate}
\end{enumerate}
\end{defi}
\begin{prop}\label{prop2_2}
Let $S$ be a non-abelian finite simple group, $G$ its universal covering group and $\ell$ a prime.
If the inductive AM-condition holds for $S$ with respect to every radical $\ell$-subgroup $D\leq S$, then $S$ satisfies the inductive AM-condition \cite[7.2]{Spaeth_AM_red} with respect to $\ell$. 
\end{prop}
The inductive AM-condition splits naturally into conditions with respect to radical $\ell$-subgroups. Nevertheless part (c) in Definition \ref{def_rel_AM} given above is different from the part (c) required in Definition 7.2 of \cite{Spaeth_AM_red}. For the proof of the proposition it is sufficient  to show the equivalence between those two versions. 

In order to prove this we need some terminology relevant for the inductive AM-condition as stated in \cite{Spaeth_AM_red}. We follow the usual conventions about projective representations (see \cite[Ch. 11]{Isa}). The trace of an endomorphism $v$ of a vector space is denoted by $\tr (v)$.

When $X_1\lhd X_2$, a $X_1${\bf -section} $\rep:  X_2/X_1\to X_2$ is any (set theoretic) section of the factor map $X_2\to X_2/X_1$ satisfying $\rep (1_{X_2/X_1})=1_{X_2}$. If moreover $X_1$ is {\em central} in $X_2$, and $\chi\in\Irr (X_2)$, a projective representation $\calQ $ of $X_2/X_1$ is said to be {\bf obtained from $\chi$ using $\rep$} if one has $\calQ = \calD\circ \rep$ where $\calD$ is a (linear) representation of $X_2$ affording the character $\chi$.

In view of \cite[7.2]{Spaeth_AM_red}, the proof of Proposition~\ref{prop2_2} clearly amounts to check the following

\begin{lem}\label{lem2_4}
Let $S$, $G$, $\ell$ and $D$ be as in Definition \ref{def_rel_AM}. Assume that $M $ is an $\Aut(G)_D$-stable subgroup of $G$ with $M \geq \NNN_G(D)$.
 Let $\chi\in\Irr_0(G\mid D)$ and $\chi'\in\Irr(M \mid D)$, such that there exist a group $A$ and characters $\w\chi$ and $\w\chi'$ satisfying the conditions of Definition~\ref{def_rel_AM_cohom}. Denote $\wh M :=M /\Z(G)$, a subgroup of $\Aut (G)$. Recall $Z:=\ker(\chi)\cap \Z(G)$, $\o G:=G/Z\lhd A$ and its subgroups $\o D:=D Z/Z$, $\o M :=M /Z$.

Let $\rep\colon  \Aut(G)_\chi \rightarrow A$ be a  $\Z (A)$-section of the canonical epimorphism $A\rightarrow \Aut(G)_\chi$. Let $\w\calD$ and $\w\calD'$ be linear representations of $A$ and $\NNN_A(\o D)$ affording $\w\chi$ and $\w\chi'$. 
Then the projective representations $\calP:=\w\calD\circ \rep$ and $\calP':=\w\calD' \circ \rep_{\wh M \NNN_A(\o D)/\Z (A)}$ satisfy the following statements 
\begin{itemize}
\item $\Res^{ \Aut(G)_\chi}_{S}(\calP )$ is obtained from $\chi$ using $\rep_{S}$, 
\item $\Res_{\wh M }^{\wh M \NNN_A(\o D)/\Z (A)}(\calP' )$ is obtained from $\chi'$ using $\rep_{\wh M  }$,
\item the factor sets $\al$ of $\calP$ and $\al'$ of $\calP'$ satisfy
\begin{equation}\label{eq_ind_AM_factor_set} 
\al' (a,a')=\al (a,a') \text{ for every }  a,a'\in  \wh M  \Aut(G)_{ D,\chi},
\end{equation} 
and
\item for every $x\in ( \wh M  \Aut(G)_{D,\chi})_{\ell'}$ with $D\Z(G)/\Z(G)\in \Syl_\ell\bigl(\Cent_{G/\Z(G)}(x)\bigr)$ we have 
\begin{equation}\label{eq_calP_calP'} \epsilon_\chi^*\tr\bigl(\calP(x)\bigr)^* =\tr\bigl(\calP'(x)\bigr)^*\end{equation}
where $ \epsilon_\chi\in \Lset{1,\ldots , \ell -1}$ is defined by $\epsilon_\chi \chi(1)_{\ell '} \equiv |G:M |_{\ell '  }\chi '(1)_{ \ell' } \mod \ell$.
\end{itemize}
\end{lem}
\begin{proof} 
Let $a,a'\in \wh M  \Aut(G)_{D,\chi} $. By the definition of $\rep$ there exists some $c\in \Cent_A(\o G)$ with $\rep(a)\rep(a')=c\rep(aa')$. 
Let $\nu\in\Irr(\Cent_A(\o G)\mid \w\chi )=\Irr(\Cent_A(\o G)\mid \w\chi' )$ (by assumption both sets coincide). 
By the definition of $\calP$ and $\alpha$ as the factor set of $\calP$ we have
$\al(a,a')=\nu(c)$.

One can argue analogously for the factor set $\al '$ of $\calP'$ and we get the same formula $\al '(a,a')=\nu (c)$. This proves Equation (\ref{eq_ind_AM_factor_set}).

Let $x\in (\wh M \Aut(G)_{D,\chi})_{\ell'}$ with $D\Z(G)/\Z(G)\in \Syl_\ell(\Cent_{G/\Z(G)}(x))$. 
Then by elementary group theory there exists some $x'\in (\rep(x) Z)\cap A_{\ell'}$ with $x'\in \Cent_A(\o D)$. Let us consider $\chi_1:=\Res_{\spann<\o G, x'>}^{A}({\w\chi})$ and  $\chi_1 ':=\Res_{\spann<\o M , x'>}^{\o M\NNN_A(\o D)}({\w\chi'})$. Note that the last character is well-defined since $x'$ centralizes $\o D$. By assumptions the blocks to which the characters $\chi_1$ and $\chi_1'$ belong have a common Brauer correspondent. 
By the definition of Brauer correspondence using central functions (see \cite{Navarro} p.87) we have 
\[ \la_{\chi_1}(\Cl_{\spann<\o G, x'>}(x')^+)^*=
\la_{\chi_1'}(\Cl_{\spann<\o M , x'>}(x')^+)^*.\]
Note that this relies on the fact that \[\Cl_{\spann<\o G, x>}(x') \cap \NNN_{\spann<\o G,x'>}(\o D)= \Cl_{\NNN_{\spann<\o G,x'>}(\o D)}(x')=
\Cl_{\spann<\o M , x'>}(x') \cap \NNN_{\spann<\o G,x'>}(\o D),\]
which in turn is a consequence of the fact that $\o D$ is the defect group of the classes containing $x'$, see also Lemma (4.16) of \cite{Navarro}. Let $\epsilon_{\chi}\in \Lset{1,\ldots , \ell -1}$ with \[ \epsilon_\chi \chi(1)_{\ell '} \equiv |G:M |_{ \ell ' }\chi ' (1)_{ \ell' } \mod \ell .\] Then the definition of $ \la_{\chi_1}$ and $
\la_{\chi_1'}$ implies 
\begin{equation} \epsilon_\chi^*\w \chi(x')^* =\w\chi'(x')^*.\end{equation}
Let $c\in\Cent_A(\o G)$ with $x'=c\rep(x)$. Note that  $\tr(\calP(x))^*= (\nu(c)^{-1}\w\chi(x'))^*$ and  $\tr(\calP'(x))^*= (\nu(c)^{-1}\w\chi'(x'))^*$. Altogether this implies Equation \eqref{eq_calP_calP'}.
\end{proof}

\begin{rem} \label{rem2_5}
 The converse of Lemma~\ref{lem2_4} is true. This can be proved using classical considerations on the Darstellungsgruppe associated to a projective representation, see the proof of \cite[11.28]{Isa} or \cite[8.28]{Navarro}. 
\end{rem}

\subsection{Inductive AM-condition for maximal defect}

It now seems natural to ask if a reduction theorem for the Alperin-McKay Conjecture relative to the class of Sylow $\ell$-subgroups  is possible. Then one considers the blocks $B$ of a group $X$ whose defect groups are Sylow $\ell$-subgroups of $X$. Those blocks are called {\bf blocks with maximal defect} in what follows. For every group $H$ we call any group $H_1/H_2$ for $H_2\lhd H_1\leq H$ a {\bf subquotient of $H$}.

We introduce a reduction theorem of the Alperin-McKay Conjecture for blocks with maximal defect. 

\begin{prop} \label{prop2_6}
Let $X$ be a finite group and $\ell$ be a prime. Assume that every non-abelian simple subquotient $S$ of $X$ with $\ell\mid |S|$ satisfies the inductive AM-condition from Definition~\ref{def_rel_AM} with respect to some Sylow $\ell$-subgroup of $S$. Then the Alperin-McKay Conjecture holds for any $\ell$-block $B$ of $X$ with maximal defect, i.e., $ |\Irr_0(B)|=|\Irr_0(b)|$, where $b$ is a Brauer correspondent of $B$.
\end{prop}
This is almost a special case of Theorem C of \cite{Spaeth_AM_red} and it can be proven along the same steps with some necessary adaptations.
We first adapt Proposition 9 of \cite{Murai04}.
\begin{prop}\label{prop2_7} Keep $X$ a finite group and $\ell$ a prime. Let $D$ a Sylow $\ell$-subgroup of $X$ and $B\in \Bl (X\mid D)$. Suppose
that the Alperin-McKay Conjecture is true  for any block with maximal defect of any group $H$ with $|H : \Z(H)| < |X : \Z(X)|$ and such that $H$ is isomorphic to a subquotient of $X/\Z(X)$. 
Then one of the following two cases holds: 
\enumroman
\begin{enumerate}
\item The Alperin-McKay Conjecture holds for $B$.
\item \label{prop2_7ii}
For any non-central normal subgroup $K$ of $X$ we have $X=K\NNN_X(D)$ and $B$ covers an $X$-invariant block of $K$ with maximal defect, in particular $X=\NNN_X(D)\operatorname{F}^*(X)$ for the generalized Fitting subgroup $\operatorname{F}^*(X)$ of $X$.
\end{enumerate}
\enumalph
\end{prop}
\begin{proof}
Let $K\lhd X$ be a non-central normal subgroup of $X$. The proof of Theorem 6 of \cite{Murai04} can be adapted and proves that by the assumption $|\Irr_0(B)|=|\Irr_0(b)|$ where $b\in\Bl(K\NNN_X(D))$ is the block with $b^X=B$. This implies $X=K\NNN_X(D)$. The considerations in the proof of Proposition 9(ii) in \cite{Murai04} can be applied to the block $B$ with maximal defect and in this case it only uses the assumptions on blocks with maximal defect given here. Moreover by Theorem (9.26) of \cite{Navarro} the block $B$ covers only blocks with maximal defect. Since the generalized Fitting subgroup $\operatorname{F}^*(X)$ is a non-central normal subgroup of $X$ it automatically satisfies $X=\NNN_X(D)\operatorname{F}^*(X)$.
\end{proof}

\renewcommand{\proofname}{Proof of Proposition \ref{prop2_6}}
\begin{proof}  {\em (Sketch)} 
According to Proposition \ref{prop2_7} we may assume that $X$ and $B$ satisfy all the statements given in \ref{prop2_7ii}. As consequence, any normal $p$-subgroup of $X$ is central.
As in Section 6 and 7 of \cite{Spaeth_AM_red} we distinguish two cases. 

Assume that there exists a normal non-central subgroup $K\lhd X$ with $K\leq \operatorname{F}^*(X)$. Let $D$ be some Sylow $\ell$-subgroup of $X$. If $K\cap D\leq \Z(X)$ then there exists a bijection between $\Irr_0(B)$ and $\Irr_0(b)$ according to \cite[6.6]{Spaeth_AM_red} where $b\in\Bl(\NNN_X(D))$ with $b^X=B$. 

Assume otherwise that the Fitting subgroup is central in $X$, and $\operatorname{E}(X)$, the group of components, is non-central. Inside $\operatorname{E}(X)$ we can take a normal subgroup $K\lhd X$ such that $K=[K,K]$ and $K/\Z(K)\cong S^r$ for a non-abelian simple group $ S$ and an integer $r\geq 1$. Since we assume that there exists no normal non-central subgroup $K_1\lhd X$ with $\ell\nmid |K_1 :(\Z(X)\cap K_1)|$ we can conclude that $\ell\mid |S|$. Note that $B$ covers a block of $K$ with maximal defect. 

By assumption, $S$ satisfies the inductive AM-condition with respect to the Sylow $\ell$-subgroups of $S$. According to its proof, the statement of Proposition 7.7 of \cite{Spaeth_AM_red} holds then for blocks of $\w G$ with maximal defect, where $\w G$ is the universal covering group of $K/\Z(K)\cong  S^r$. 
Furthermore we can conclude that also the statement of Theorem 7.9 of \cite{Spaeth_AM_red} holds for blocks with maximal defect, in particular it proves that we can apply it to our $B\in\Bl(X)$.
As in the proof of Theorem C of \cite{Spaeth_AM_red} (page 183) we can conclude that $|\Irr_0(B)|=|\Irr_0(b)|$, where $b\in\Bl(\NNN_X(K\cap D))$ is such that $b^X=B$. By the inductive assumption this implies the Alperin-McKay Conjecture for $B$. 
\end{proof}
\renewcommand{\proofname}{Proof}

\section{Clifford theory in covering blocks}\label{sec2+}

For the proof of Theorem \ref{thm2_7} below we use results from \cite{KoshitaniSpaeth}. First we recall the definition of Dade's {\bf ramification group} from \cite{Dade} in the reformulation given by \cite{Murai_Dade_group}.  

\begin{defi}\label{H[b]}  Whenever $N\lhd A$ are finite groups and $e$ is an $\ell$-block of $N$, one denotes by $A[e]$ the group with $N\leq A[e]\lhd A_e\leq A$ defined by 
\[A[e]=\{ a\in A_e\mid \la_{e^{(a)}}(\Cl_{\spann<N,a>}(y)^+)\not= 0 \text{ for some }y\in aN\},\]
 where in the above $e^{(a)}$ denotes an arbitrary block of $\spann<N,a>\leq A_e$ covering $e$.
(It is known that the above definition does not depend on the choice of the block $e^{(a)}$, see \cite[3.3]{Murai_Dade_group}.)
\end{defi}

The following technical lemma can be proved now and is used later. 
\begin{lem} \label{red}
Let $N\lhd A$ and $N\leq J\lhd A$ such that $A/J$ and $J/N$ are abelian. 
Let $b\in\Bl(N)$ and $\chi, \phi \in \Irr(b)$. Let  $\w\chi\in \Irr(A)$ and $\w\phi\in\Irr(A[b])$ be  extensions of $\phi$ and $\chi$, respectively with $\bl\bigl(\Res^A_{J_1}(\w\chi)\bigr) =\bl\bigl(\Res^{A[b]}_{J_1}(\w\phi)\bigr)$ for every $J_1$ with $N\leq J_1 \leq J[b]$. 
Then there exists an extension $\w\chi_1$ of $\Res^A_{J}(\w\chi)$ to $A$ with 
$\bl\bigl(\Res^A_{J_2}(\w\chi_1)\bigr)=\bl\bigl(\Res^{A[b]}_{J_2}(\w\phi)\bigr)$ for every $J_2$ with $N\leq J_2\leq A[b]$.
\end{lem}
\begin{proof} 
Since $A/N$ is solvable, there exists some group $I$ with $N\leq I\leq A[b]$ such that $I/N$ is a Hall $p'$-subgroup of $A[b]/N$ and $(I\cap J )/N$ is a Hall $p'$-subgroup of $J[b]/N$, see \cite[18.5]{A}. 
According to Theorem C(b)(1) of \cite{KoshitaniSpaeth} there exists an extension $\w\chi_2\in\Irr(I)$ of $\chi$ to $I$ with $\bl(\w\chi_2)= \bl(\Res^{A[b]}_{I}(\w\phi))$.
This extension satisfies also $\bl(\Res^{I}_{I\cap J}(\w\chi_2))=\bl(\Res^{A[b]}_{I\cap J}(\w\phi))$ according to Lemmas 2.4 and 2.5 of \cite{KoshitaniSpaeth}.
By Lemma 3.7 of \cite{KoshitaniSpaeth} there is a unique character in $\Irr(I\cap J\mid \chi)$ with this property, hence  $\Res^{A}_{I\cap J}(\w\chi)=\Res^{A[b]}_{I\cap J}(\w\chi_2)$ by the assumptions on $\w\chi$.

By Lemma 5.8(a) of \cite{CabSpaeth2} there exists an extension $\eta\in\Irr(IJ)$ of $\Res^{A}_{J}(\w\chi)$ such that $\Res^{IJ}_I (\eta )=\w\chi_2$.
Since $\Res^{A}_{J}(\w\chi)$ extends to $A$ and $A/J$ is abelian, every character of $\Irr(A\mid\Res^{A}_{J}(\w\chi))$ is an extension of $\Res^{A}_{J}(\w\chi)$.
Since $\Irr(A\mid \eta)\subseteq\Irr(A\mid \Res^{A}_{J}(\w\chi))$ there exists an extension $\w\eta$ of $\eta$ to $A$. 
According to \cite[2.4]{KoshitaniSpaeth} (for ordinary characters instead of Brauer characters) the character $\w \eta$ satisfies then
\[ \bl\bigl(\Res_{\spann<N,x>}^A( \w\eta)\bigr)=\bl\bigl(\Res^{A[b]}_{\spann<N,x>}( \w \phi)\bigr) \text{ for every $\ell$-regular }x \in I.\]
Note that every $\ell$-regular $x\in A[b]$ is conjugate to some element in $I$. Accordingly the above equality holds for every $\ell$-regular $x\in A[b]$. According to \cite[2.5]{KoshitaniSpaeth} this implies
\begin{align*}
 \bl\bigl(\Res^A_{J_2} (\w \eta)\bigr)&=\bl\bigl(\Res^{A[b]}_{J_2}( \w \phi)\bigr) \text{ for every group $J_2$ with }N\leq J_2\leq A[b].
 \qedhere
 \end{align*}
\end{proof}

\section{A criterion with stability conditions}\label{sec3}
Here we prove a criterion for inductive AM-condition adapted to the context of simple groups of Lie type. It is related to the one introduced in \cite[2.10]{Spaeth5} and was also used in \cite{CabSpaeth2}. In order to deal with block theoretic properties we have to make additional assumptions that limit applications to a certain class of blocks. 


\begin{theorem}\label{thm2_7}
Assume the situation of \cite[2.10]{Spaeth5}: 
Let $S$ be a finite non-abelian simple group and $\ell$ a prime with $\ell\mid |S|$. Let $G$ be the universal covering group of $S$ and $D$ a radical $\ell$-subgroup of $G$. Assume we have a semi-direct product $\w G\rtimes E$ and a subgroup $M<G$ such that the following statements hold: 
\enumroman
\begin{enumerate}
\item \label{thm2_7i}
\begin{itemize}
\item  $G=[\w G,\w G]$ and
  $E$ is abelian,
\item $\Cent_{\wG\rtimes E}(G)= \Z(\wG)$ and $\w GE/\Z(\wG)\cong\Inn(G)\Aut(G)_D$ by the natural map,
\item $\NNN_G(D)\leq M\neq G$ and $M$ is $\Aut(G)_D$-stable,
 	\item any element of $\Irr_0(G\mid D)$ extends to its stabilizer in $ \w G$,
\item any element of $\Irr_0(M\mid D)$ extends to its stabilizer in  $ \w M :=M\NNN_{\wG}(D)$.
\end{itemize} 
\item \label{thm2_7ii}
 For $ \calG:=\Irr\left (\Gu\mid \Irr_0(G\mid D)\right )$ and $\calM:=\Irr\left (\Mu\mid \Irr_0(M\mid D)\right )$ there exists an $\NNN_{\wG E}(D)$--equivariant bijection \[\w \Omega_D: \calG \longrightarrow \calM\] with 
 \begin{itemize}
	\item $ \Omegau_D\bigl(\calG\cap\Irr(\Gu\mid \nu)\bigr)= \calM\cap\Irr(\Mu\mid \nu)$ for every $\nu \in \Irr(\Z(\Gu))$,
	\item \label{Omega_u_epsilon_equiv2_block} $\bl\bigl(\Omegau_D(\chi)\bigr)^{\wG}=\bl(\chi)$ for every $\chi\in\calG$, and
	\item \label{Omega_u_epsilon_equiv2_8} $\Omegau_D(\chi\mu)= \Omegau_D(\chi)\ \Res^{\wG}_{\Mu}( \mu)$ for every $\mu \in \Irr(\Gu/G)$ and  $\chi\in\calG$.
\end{itemize}
\item \label{thm2_7iii}
For every $\chi\in \calG$ there exists some $\chi_0\in \Irr(G\mid \chi)$ such that 
\begin{itemize}
\item $(\Gu  E)_{\chi_0}= \Guchinull  E_{\chi_0}$, and
\item $\chi_0$ extends to $G   E_{\chi_0}$.
\end{itemize} 
\item \label{thm2_7iv}
For every $\psi\in \calM$ there exists some $\psi_0\in \Irr(M\mid \psi)$ such that 
	\begin{itemize}
		\item $O= (\Gu\cap O) \rtimes (E\cap O)$ for $O:=G(\Gu\rtimes E)_{D,\psi_0}$, and
		\item $\psi_0$ extends to $M(G\rtimes E)_{D,\psi_0}$.
	\end{itemize}
	\item \label{thm2_7_block} For every group $J$ with $G\leq J\leq \wG$ and any $p$-subgroup $\w D\leq J$ with $\w D\cap G=D$, every $b\in\Bl(J\mid \w D)$ satisfies $\wG_b=\wG$.
\end{enumerate}
\enumalph
Then the inductive AM-condition holds for $S$ with respect to $D\Z(G)/\Z(G)$ in the sense of Definition~\ref{def_rel_AM}.
\end{theorem}

Apart from the block statements of \ref{Omega_u_epsilon_equiv2_block} and  \ref{thm2_7_block}, note that in the case of $D\in\Syl_{\ell}(G)$ the above assumptions coincide with the ones of \cite[2.10]{Spaeth5}. Accordingly many ideas of its proof can be transferred, especially the bijection $\Omega_D$ is constructed in the same fashion. The main difficulty consists in proving the additional condition on the blocks from Definition \ref{def_rel_AM_cohom_iv}.

\subsection{The bijection $\Omega_D$}\label{sec3_1}
Let $\calG_0$ be a complete representative system in $\calG$ under the action of $\Irr(\wG/G)\rtimes E$, where $\Irr(\wG/G)$ acts on $\calG$ by multiplication. Let also $\calM_0:=\Omegau_D(\calG_0)$, and $\calG'_0$ a subset of $\Irr_0(G\mid D)$ such that for every $\chi\in \calG$ the set $\calG'_0\cap \Irr_0(G\mid \chi)$ contains exactly one character $\chi_0$ and this character has the properties from \ref{thm2_7iii}. Analogously one can choose a set $\calM'_0\subseteq\Irr_0(M\mid D)$ such that for every $\psi\in \calM$ the set $\calM'_0\cap \Irr_0(M\mid \psi)$ contains exactly one character $\psi_0$ and this character has the properties from \ref{thm2_7iv}. The proof of Theorem 2.10 of \cite{Spaeth5} shows that we obtain an $\NNN_{\wG E}(D)$-equivariant bijection 
$\Omega_D: \Irr_0(G\mid D)\longrightarrow \Irr_0(M\mid D)$ 
by setting first
\[ \Omega_D(\chi_0)=\psi_0,\]
whenever $\chi_0\in\calG'_0$, and $\psi_0\in\calM'_0$, with $\Omegau_D(\Irr(\w G\mid\chi_0)\cap \calG_0 )= \Irr(\w M\mid\psi_0)\cap \calM_0$ and setting additionally 
\[ \Omega_D(\chi_0^a):=\Omega_D(\chi_0)^a \text{ for every }\chi_0\in\calG'_0 \text{ and } a\in\NNN_{\wG E}(D).\]

Because of the property of $\Omegau_D$ stated in \ref{thm2_7ii} our bijection $\Omega_D$ satisfies \[\Omega_D\bigl(\Irr_0(G\mid D)\cap \Irr(G\mid \nu)\bigr)\subseteq \Irr(M\mid \nu) \text{ for every }\nu\in\Irr(\Z(G)).\]
Let $\chi_0\in\calG'_0$. Then there exist a unique character $\chi\in \Irr(\w G\mid\chi_0)\cap \calG_0$ and a unique character $\psi\in \Irr(\w M\mid\psi_0)\cap \calM_0$. By construction they satisfy $\psi :=\w\Omega_D(\chi )$.   Assumption \ref{Omega_u_epsilon_equiv2_8} gives $\bigl(\bl (\w\Omega_D (\chi ))\bigr)^\wG =\bl (\chi )$. Clearly $\bl (\chi )$ covers $\bl (\chi_0)$ and $\bl \bigl(\w\Omega_D (\chi )\bigr)$ covers $\bl \bigl(\Omega_D (\chi_0)\bigr)$ by \cite[9.2]{Navarro}. Moreover $\bl \bigl(\Omega_D (\chi_0 )\bigr)^G$ is covered by $\bl (\chi )=\bigl(\bl (\w\Omega_D (\chi ))\bigr)^\wG$ thanks to the Harris-Kn\"orr theorem \cite[9.28]{Navarro} applied to $G\lhd \w G$.

So we get that both $\bl \bigl(\Omega_D (\chi_0 )\bigr)^G$ and $\bl (\chi_0 )$ are covered by $\bl (\chi )$, therefore $\bl \bigl(\Omega_D (\chi_0 )\bigr)^G$ and $\bl (\chi_0 )$ are $\wG$-conjugate by \cite[9.3]{Navarro}). But now assumption \ref{thm2_7_block} implies that $$\bl \bigl(\Omega_D (\chi_0 )\bigr)^G$$ and $\bl (\chi_0 )$ are $\w G$-invariant and hence they coincide.

Since $\calG'_0$ is an $\NNN_{\wG E}(D)$-transversal in $\Irr_0(G\mid D)$ this also implies that 
\begin{align}\label{star} \bl\bigl(\Omega_D(\chi)\bigr)^G&=\bl(\chi) \text{ for every }\chi\in\Irr_0(G\mid D).\end{align}

In all we have constructed an $\Aut(G)_D$-equivariant bijection $\Omega_D: \Irr_0(G\mid D) \rightarrow \Irr_0(M_D\mid D)$ such that \[ \bl\bigl(\Omega(\chi)\bigr)^G=\bl(\chi) \text{ for every } \chi\in\Irr_0(G\mid D).\]


\subsection{ }\label{sec3_2}
It remains to ensure that any $\chi\in\Irr_0(G\mid D)$ satisfies the condition stated in Definition \ref{def_rel_AM_cohom}. According to Lemma \ref{lem2_4} and Remark \ref{rem2_5} this is equivalent to the fact that $\chi$ satisfies 7.2(c) of \cite{Spaeth_AM_red}. But it is clear that if this condition holds for a character $\Irr_0(G\mid D)$ then it also holds for any $\Aut(G)_D$-conjugate of this character. 
Hence we may concentrate on characters $\chi\in\calG_0'\subseteq\Irr_0 (G\mid D)$. 

For any character $\chi\in\calG_0'$ we prove in a first step the following statement that already provides a group $A$ and characters with some of the required properties. We continue using the notation from Sect. \ref{sec3_1}. 

\begin{prop}\label{H[b]2H}
Let $\chi\in\calG_0$ and $\Omega_D$ be the bijection from Sect. \ref{sec3_1}. Then there exists a group $A$, characters $\w \chi$ and $\w\psi$ such that 
\enumroman
\begin{enumerate}
\item for $Z:=\ker(\chi)\cap \Z(G)$ and $\o G:=G/Z$, $A$ satisfies $\o G\lhd A$, $\Cent_A(\o G)=\Z(A)$ and $A/\Z(A)\cong \Aut(G)_\chi$,
\item $\w\chi\in\Irr(A)$ is an extension of the character $\o\chi\in\Irr(\o G)$ that lifts to $\chi$
\item $\w\psi\in\Irr\bigl(\o M \NNN_A(\o D)\bigr)$ is an extension of the character $\o\psi \in \Irr(\o M)$ that lifts to $\psi:=\Omega_D(\chi)$ where $\o M:=M/Z$ and $\o D:=DZ/Z$,
\item $\Irr ({\Z(A)}\mid \w\chi)= \Irr ({\Z(A)}\mid \w\psi)$.

\item there exists a group $J$ with $\o G \Z (A)\leq J \lhd A$ with abelian factor groups $A/ J$ and $J/\o G$, such that \[\bl\bigl(\Res^{A}_{J_2}(\w\chi)\bigr)=\bl\bigl(\Res^{\o M \NNN_A(\o D)}_{\o M \NNN_{J_2}(\o D)}(\w\psi)\bigr)^{J_2}\text{ for every $J_2$ with }\o G\leq J_2 \leq J.\]
\end{enumerate}
\enumalph
\end{prop}

\begin{proof}
By definition the elements $\chi\in\calG_0'$ satisfy 
\[( {\wG E})_\chi=\wG_\chi\rtimes E_\chi.\]
Recall $\w M:=M\NNN_{\w G}(D)$. The characters $\chi$ and $\psi$ have unique extensions $\w\chi_1\in\Irr(\wG_\chi)$ and $\w\psi_1\in\Irr(\w M_\psi)$ with $\Ind^\wG_{\wG_\chi}(\w\chi_1)\in \calG_0$ and $\Ind^{\wM}_{\wM_\psi} (\w\psi_1)\in\calM_0$. Using the considerations that lead to \ref{star} above, we get 
\[ \bl\bigl(\Ind^\wG_{\wG_\chi}(\w\chi_1)\bigr)=\bl\bigl(\Ind^{\wM}_{\wM_\psi} (\w\psi_1 )\bigr)^\wG.\]
The assumption \ref{thm2_7_block} and the Harris-Kn\"orr Theorem from \cite[9.28]{Navarro} with \cite[9.3]{Navarro} imply that  
\begin{align}\label{blockswchi1}
\bl\bigl(\Res^{\wG_\chi}_{J_2}(\w\chi_1)\bigr)&=\bl\bigl(\Res^{\w M_\psi}_{J_2\cap \wM_\psi}(\w\psi_1)\bigr)^{J_2}\forevery J_2 \with G\leq J_2\leq \w G_\chi.
\end{align}

Let $\rep: \Aut(G)_\chi \longrightarrow ({\wG E})_\chi$ be a $\Z(\w G)$-section. Let $\calP_1$ be a projective representation of $\w G_\chi/\Z(\w G)$ obtained from $\chi_1$ using $\rep_{\w G_\chi/\Z(\w G)}$ and $\calP'_1$ be a projective representation of $\w M_\chi/\Z(\w G)$ obtained from $\psi_1$ using $\rep_{\w G_\chi/\Z(\w G)}$. 

Recall that quotients of subgroups of $G$ over $Z_1:=\Z(G)$ are canonically identified with subgroups of $\Aut(G)$, and analogously quotients of subgroup of $\w G$ over $\w Z_1:=\Z(\w G)$.
According to the proof of \cite[2.7]{Spaeth5} there exist projective representations  $\calP$ of $\Aut(G)_\chi$ and $\calP'$ of $ (M/Z_1) \Aut(G)_{D,\chi}$ such that the factor sets $\beta$ and $\beta'$ of $\calP$ and $\calP'$ have finite order and satisfy 
\[ \beta'(a,a')=\beta(a,a') \forevery a,a'\in (M/Z_1) \Aut(G)_{D,\chi}\] 
and 
\[ \calP(g)=\calP_1(g) \und \calP'(m)=\calP'_1(m) \forevery g\in \w G_\chi/\w Z_1 \und m\in \w M_\psi/\w Z_1.\] 

The factor set $\beta$ determines a group $A$ as in the proof of \cite[8.28]{Navarro}: Let $C$ be the cyclic subgroup of $\CC^*$ generated by the values of $\beta$, and let $ A$ be the group whose elements are $(a,\zeta)$ ($a\in\Aut(G)_\chi$ and $\zeta\in C$), and whose multiplication is given by 
\[ (a,\zeta) (a',\zeta')= (aa',\zeta\zeta'\beta(a,a')) \forevery a, a'\in\Aut(G)_\chi \und \zeta ,\zeta '\in C.\]
The epimorphism $\epsilon:A\rightarrow \Aut(G)_\chi$ with  $(a,\zeta)\mapsto a$ is a central extension with kernel $C$ and $\o G\lhd A$ by the proof of \cite[2.8]{Spaeth5}. Let us denote by $\Aut_{\wG_\chi}(G)$ the group of automorphisms of $G$ induced by $\wG_\chi$. Now by the definition of $A$, the projective representation $\calP$ lifts to a linear representation $\calD$ of $A$, that is defined by 
\[ \calD(a,\zeta)=\zeta\calP(a)\forevery a\in\Aut(G)_\chi \und \zeta\in C.\]
Then the character $\w\chi_2$ afforded by $\calD$ satisfies
\[ \w\chi_2(a,\zeta)=\zeta\tr(\calP(a))\text{ for every } a \in \Aut(G)_\chi \und \zeta\in C.\]
For elements $a\in\Aut_{\wG_\chi}(G) $ this implies 
\begin{align}\label{eqwchi0}
 \w\chi_2(a,\zeta)=\zeta\w\chi_1(\rep(a)) \text{ for every } a \in \Aut_{\wG_\chi}(G)  \und \zeta\in C.\end{align}

Let $D_2\in\Syl_\ell(C\spann<(d,1)\mid d\in DZ_1/Z_1>)$ and $M_2:= \spann<(m,1)\mid m\in M/Z_1>$. 
Since $\beta'(a,a')= \beta(a,a')$ for every $a,a'\in (M/Z_1) \Aut(G)_{D,\chi}$, $\calP'$ lifts to the linear representation $\calD'$ of $ M_2\NNN_A(D_2)$ that is defined by 
\[ \calD'(a,\zeta)=\zeta\calP'(a)\forevery a\in (M/Z_1)\Aut(G)_{\chi,D} \und \zeta\in C.\]
Let $ J:=\epsilon^{-1}(\Aut_{\wG_\chi}(G))$ and $H:=\epsilon^{-1}(  (M/Z_1) \Aut(G)_{D,\chi})$.
Like above let $\w\psi_2\in\Irr(M_2 \NNN_A(D_2))$ be the character afforded by $\calD'$. Then $\w\psi_2$ satisfies 
\[ \w\psi_2(a,\zeta)=\zeta\tr(\calP'(a))\text{ for every } a \in \Aut(G)_{D,\chi} \und \zeta\in C, \]
and hence 
\begin{align}\label{eqwpsi0}
 \w\psi_2(a,\zeta)=\zeta\w\psi_1(\rep(a))\text{ for every } a \in (M/Z_1)\bigl(\Aut_{\w G_\chi}(G)\bigr)_{D} \und \zeta\in C. \end{align}

Now we consider restrictions $\wchi_2$ and $\w\psi_2$. By the equation \eqref{eqwchi0} and \eqref{eqwpsi0} the equality \eqref{blockswchi1} leads to 
\begin{align*}
 \la_{\Res^J_{\spann<\o G,x>}(\w\chi_2)}(\Cl_{\spann<\o G,x>}(x)^+)&=\zeta  \la_{\Res^J_{\spann<\o G, \rep(a)>}(\w\chi_1)}\left (\Cl_{\spann<\o G, \rep(a)>}(\rep(a))^+\right )\\
 &=
\zeta \la_{\Res^J_{\spann<\o G, \rep(a)>\cap \w M} (\w\psi_1)}\left((\Cl_{\spann<\o G, \rep(a)>}(\rep(a))  \cap \w M)^+\right )\\
&=\la_{\Res^J_{\spann<\o G,x>\cap H}(\w\psi_2)}\left ((\Cl_{\spann<\o G,x>}(x)\cap H)^+\right )
\end{align*}
for every $x=(a,\zeta)\in J$. This implies 
\[ \bl\bigl(\Res^{A}_{J_2}(\w\chi_2)\bigr)= \bl\bigl(\Res^{H}_{J_2\cap H}(\w\psi_2)\bigr)^{J_2} 
\text{ for every } J_2 \text{ with } \o G\leq J_2\leq J,
\]
according to \cite[2.5(b)]{KoshitaniSpaeth}.
This proves that the proposition, in particular (v), holds with $\w\chi_2$ as $\w\chi$ and $\w\psi_2$ as $\w\psi$. 
 \end{proof}
 
We will now combine Proposition~\ref{H[b]2H} together with Lemma \ref{red} to finish the proof of the theorem of this section.

\renewcommand{\proofname}{Proof of Theorem \ref{thm2_7}}

\begin{proof} 
The bijection from \ref{sec3_1} satisfies the requirements of Definition \ref{def_rel_AM} with respect to $D\Z(G)/\Z(G)$, apart possibly the conditions \ref{def_rel_AM_cohom}. As explained above according to Lemma \ref{lem2_4} and Remark \ref{rem2_5}, it is sufficient to verify that $\chi\in\calG_0$ satisfies \ref{def_rel_AM_cohom}. 

For $\chi\in\calG_0$ Proposition \ref{H[b]2H} applies. We continue using the notation introduced in its proof. Let $H:=M_2\NNN_A(D_2)$ and $\w\psi\in \Irr (H)$ the character from Proposition \ref{H[b]2H}. Then by Theorem~C in \cite{KoshitaniSpaeth}, there exists $\w\phi\in \Irr (A[b])$ such that $ \Res^{A[b]}_{\o G}(\w\phi )$ is irreducible and 
\[ \bl \bigl(\Res^{H}_{H \cap J_2}(\w\psi)\bigr)^{J_2}=\bl \bigl(\Res^{A[b]}_{J_2}(\w\phi)\bigr) \forevery J_2 \with \o G\leq J_2 \leq A[b],\]
where $b=\bl \bigl(\Res^{A[b]}_{\o G}(\w\phi)\bigr)$. 

Together with Proposition \ref{H[b]2H} this proves that the assumptions of Lemma \ref{red} are satisfied for $\w\chi$ and $\w\phi$ (and $N=\o G$). Accordingly there exists an extension $\w\chi_3\in\Irr(A)$ of $\o\chi\in\Irr(\o G)$ with $\Res^A_{J}(\w\chi_3)=\Res^A_{J}(\w\chi)$  and 
\[ \bl\bigl(\Res^{A[b]}_{J_2}(\w\phi)\bigr)=\bl\bigl(\Res^A_{J_2}(\w\chi_3)\bigr) \forevery J_2 \with \o G\leq J_2\leq A[b].\]
By the definition of $\w\phi$ this leads to 
\[ \bl \bigl(\Res^{H}_{J_2 \cap H}(\w\psi)\bigr)^{J_2}= \bl\bigl(\Res^{A[b]}_{J_2}(\w\phi)\bigr)=\bl\bigl(\Res^A_{J_2}(\w\chi_3)\bigr) \forevery J_2 \with \o G\leq J_2\leq A[b].\]

For $J_3$ with  $\o G\leq J_3\leq A$ and $J_2:=J_3\cap A[b]$, 
we know from \cite{Murai_Dade_group} that 
\[ \bl\bigl(\Res^{H}_{J_3\cap H}(\w\psi)\bigr)=
\bl\bigl(\Res^{H}_{J_2\cap H}(\w\psi)\bigr)^{J_3\cap H}\]
and \[ 
\bl\bigl(\Res^{A}_{J_3}(\w\chi_3)\bigr)=
\bl\bigl(\Res^{A}_{J_2}(\w\chi_3)\bigr)^{J_3}.\]
Applying transitivity of block induction now gives us 
\[\bl \bigl(\Res^A_{J_3}(\w \chi_3)\bigr)=  \bl \bigl(\Res^H_{H\cap J_3}(\w \psi)\bigr)^{J_3}.\] 
Accordingly the group $A$ and the characters $\w\chi_3$ and $\w\psi$ satisfy the conditions from \ref{def_rel_AM_cohom}. \end{proof}

\renewcommand{\proofname}{Proof}


\section{Properties of blocks in type $\tA$}\label{sec_blocks}
In this section we show that the assumptions of Theorem~\ref{thm2_7} are satisfied by blocks of maximal defect of $G=\SL_n( q)$ or $\SL_n(-q):=\SU_n(q)$. This is based on some results from \cite{CabSpaeth2}, whose notations are used in the following.

We take $n\geq 2$, $m\geq 1$, $\ep\in\{ 1,-1\}$, $q=p^m$ the power of a prime $p$. We denote $\w\bG =\GL_n (\o\FF_q)\geq \bG =\SL_n (\o\FF_q) $ with Frobenius endomorphism $F = \gamma^{1-\ep\over 2}\circ F_p^m$ where $\gamma$ is the automorphism sending any matrix to its transpose-inverse. We denote the corresponding finite groups as $\wG =\GL_n (\ep q)=\w\bG^F\geq G =\SL_n (\ep q) =\bG^F$.

Let $\ell$ be a prime with $\ell\neq p$. One denotes by $d\geq 1$ the order of $q$ mod $\ell$ when $\ell$ is odd, while $d$ is the order of $q$ mod $4$ when $\ell = 2$.

We choose $\bS$ to be a Sylow $\Phi_d$-torus of $(\bG ,F)$. Moreover let $M=\NNN_\bG (\bS )^F$, $\w M=\NNN_\wbG (\bS )^F$, $\bC=\Cent_\bG (\bS)$ and $\w\bC=\Cent_\wbG (\bS )$.


First we show that the map $\w\Omega\colon \Irr (\w G\mid\Irr_\lp (G))\to  \Irr (\w M\mid\Irr_\lp (M))$ defined in \cite[\S 6]{CabSpaeth2} is compatible with Brauer induction (as required in Theorem~\ref{thm2_7ii}). 
In \cite[7.1]{CabSpaeth} a bijection of a similar nature with such a property has been constructed.

In the following let $\w\calM_M$ and  $\w\calM_G$ be the sets of parameters and the associated maps $\psi^{(G)}\colon \w\calM_G\to\Irr (\wG\mid\Irr_\lp (G))$ and $\psi^{(M)}\colon \w\calM_M\to\Irr (\w M \mid\Irr_\lp (M))$ defined as in \cite[\S 6]{CabSpaeth2} and inducing parametrizations of the sets of characters by rational classes of $\w\calM_M$ and  $\w\calM_G$ (see \cite[6.3]{CabSpaeth2}).

\begin{prop} \label{prop_block}  Let $(s,\la , \eta )\in\calM_M$. Moreover let $B\in \Bl(\w G)$ and  $b\in\Bl(\w M)$, such that $\psi^{(G)}{(\bS,s,\la ,\eta)} \in\Irr (\w G)$ belongs to $B$ and $\psi ^{(M)}{(s,\la ,\eta)}\in\Irr (\w M)$ belongs to $b$. 
Then \begin{enumerate}
\item $\bC =\Cent_\bG (\Zent(\bC )^F_\ell )$ and $\w\bC^F =\Cent_\wbG (\Zent(\w\bC )^F_\ell )^F$  

\item $B$ and  $b$ have the same Brauer correspondent. 
\end{enumerate}
\end{prop} 
 
\begin{proof} The first equality in (a) clearly implies the second. 

Proving $\bC =\Cent_\bG (\Zent(\bC )^F_\ell )$ is easy by describing explicitly $\bC$ and the action of $F$ on it (see proof of \cite[5.9]{Ma06} or the description of \cite[4.10.2]{GLS3}). Note also that when $\ell$ is odd, our claim is implied by \cite[22.17.(ii)]{CabEng}, and when $\ell = 2$ by Proposition~8 of \cite{Eng}, both in the form $\bC = \Cent^\circ_\bG (\Z (\bC )_\ell^F)$. This gives our claim since all centralizers of toral subgroups are connected in $\bG=\SL_n(\FF)$ and $\wbG:=\GL_n(\FF)$ (see \cite[14.16.(a)]{MalleTesterman}).

The considerations of \cite[7.1]{CabSpaeth} apply in this situation and prove part (b).
\end{proof}

Now in order to prove the assumption in Theorem~\ref{thm2_7_block} we apply the following general observation. 
\begin{prop}\label{CGCJ}
Let  $J\lhd \w G$, $B\in\Bl(J)$ and $D$ a defect group of $B$. If $\Cent_{\w G}(\Cent_J(D)_\lp) G=\w G$, then $\wG_B=\wG$. 
\end{prop}
\begin{proof}
 Let $b\in\Bl(\Cent_G(D)D)$ be a block with $b^G=B$. Such a block exists according to Theorem 4.14 of \cite{Navarro}. Let $c\in \Cent_{\w G}(\Cent_G(D)_{\lp})$. Then the homomorphisms $\la_{b^c}: \Z\bigl(\o\FF_\ell (\Cent_G(D)D)\bigr)\rightarrow \o\FF_\ell$ and $\la_b:\Z\bigl(\o\FF_\ell (\Cent_G(D)D)\bigr)\rightarrow \o\FF_\ell$ associated to $b$ and $b^c$ as in \cite[Section 3]{Navarro} coincide on the class sums of $\ell$-regular elements. According to \cite[Ex. 4.5]{Navarro} this implies $b^c=b$. For every $g\in \w G$ there exists some $c\in \Cent_G(\Cent_{\w G}(D)_\lp)$ with $g G=cG$ because of the assumption 
$\Cent_{\w G}(\Cent_G(D)_{\lp}) G=\w G$. This element $c$ satisfies $B^g=B^c$. But on the other hand  $(b^G)^c=(b^c)^G$. The latter block $(b^c)^G$ coincides with $b^G$. This implies $B^c=B$. 
\end{proof}

\begin{prop}\label{prop3_5} Let $Q$ be a Sylow $\ell$-subgroup of $G$ normalizing $\bS$ (see \cite[3.4(4)]{BrMa}). Then $\Cent_{\w G}(\bS)$ contains $\Cent_{\w G}(Q)_{\ell '}$ and $\Cent_{\w G} (\Cent_{\w G}(\bS))G = \w G$.
 \end{prop}
 
 \begin{proof} Denote $\bC:=\Cent_\bG (\bS )$ which is an $F$-stable Levi subgroup of $\bG$. Note that $Q$ normalizes $\bC$ and its centre, so we clearly have  $Q\geq\Z (\bC )_\ell^F$. Then our first statement is obviously a consequence of $\bC = \Cent_\bG (\Z (\bC )_\ell^F)$ which is contained in Proposition~\ref{prop_block}(a) above. 
 
For the last statement, we prove $\Cent_{\w\bG} (\Zent (\wbG )\bC )^F\bG^F=\w\bG^F$ which is slightly stronger than our claim. We have $\w\bG =  \Cent_{\w\bG} (\bC )\bG$ since $\w\bG =\Z (\w\bG )\bG$. Then we get our claim as a standard consequence of Lang's theorem (see for instance \cite[8.1.(i)]{CabEng}) as soon as we check that $\bG\cap\Cent_{\w\bG} (\bC )$ is connected. The latter is $\Z (\bC )$ since $\Cent_\bG (\bC )\leq \Cent_\bG ( \bS)=\bC$. Then our claim is contained in the fact that $\Zent (\bC)$ is connected for which we refer to the proof of \cite[5.6(a)]{CabSpaeth2}. \end{proof}

\begin{prop}\label{wG_b} If $G\leq J\leq \wG$ and $b\in\Bl(J)$ has a defect group $\w D$ such that $\w D\cap G$ is a Sylow $\ell$-subgroup of $G$, then $\wG_b=\wG$.

 \end{prop}
 
 \begin{proof} We apply Proposition~\ref{CGCJ}. So it suffices to check that $\Cent_\wbG (\Cent_J(\w D)_\lp )^F \bG^F =\wbG^F$. Let $\bS$ be a Sylow $\Phi_d$-torus of $(\bG ,F)$.
 We know that $\NNN_\wbG (\bS )^F$ contains a Sylow $\ell$-subgroup of $\wbG^F$ (see \cite[3.4(4)]{BrMa}), so we can assume that $\w D$ normalizes $\bS$. Now Proposition~\ref{prop3_5} gives $\Cent_J(\w D)_\lp\leq \Cent_\wG (\w D\cap G)\leq \Cent_\wbG (\bS )^F$, so $\Cent_\wbG (\Cent_J(\w D)_\lp )^F\bG^F\geq \Cent_\wbG (\Cent_\wG (\bS))^F\bG^F$.
 The latter is $\wG$ by Proposition~\ref{prop3_5}.
  \end{proof}

\section{Proof of Theorem \ref{thm_PSL_ist_AM_gut}}\label{EndProof}

We now conclude our proof of Theorem \ref{thm_PSL_ist_AM_gut}. More precisely, we prove

\begin{theorem}\label{thm_PSL_ist_AM_gut+}
The inductive AM-condition from Definition \ref{def_rel_AM} above holds for the simple groups $S\in\{\PSL_n(\pm q)\}$ with respect to their Sylow $\ell$-subgroups, where $\ell$ is any prime.
\end{theorem}

The proof uses the criterion given in Theorem~\ref{thm2_7} but we must review the exceptions in low ranks where the choice of $G$ or $M$ is not the generic one. This can be caused by exceptional Schur multipliers. On the other hand the structure of the normalizer of the Sylow subgroup can be non generic in low rank. Then we apply the results of Sect.~\ref{sec_blocks}.


\subsection{Exceptional Schur multipliers.}\label{ExcMul}
 If $S=\PSL_n(\ep q)$ is simple and not among the groups \[ \{  \PSL_2(4) , \PSL_3(2)\cong \PSL_2(7)   ,  \PSL_2(9)  , \PSL_4(\pm 2) , \PSL_3(4) , \PSL_6(-2) ,  \PSL_4(-3)    \} ,\] its universal covering is the corresponding $G=\SL_n(\ep q)$ with the Schur multiplier of $S$ being a cyclic group of order $(n,q-\ep )$ (see \cite[Table 24.3]{MalleTesterman}). The eight cases listed above correspond to simple groups which have a so-called {\em exceptional} Schur multiplier. For them
 the inductive AM-condition from \cite[7.2]{Spaeth_AM_red}  has been verified for $S$ and all primes by Th. Breuer, see \cite{Breuer}. 
They could also be treated in the same spirit as \cite[\S 7]{CabSpaeth2}.

So we may now assume that $n$ and $\ep q$ are such that $\PSL_n(\ep q)$ is simple and $\SL_n(\ep q)$ is its universal covering group.

\subsection{Defining characteristic.}\label{DefCha} We assume that $\ell $ divides $q$. Then a Sylow $p$-subgroup is the unipotent upper triangular subgroup of $G$ whose normalizer is the upper triangular subgroup of $G$. We take the latter as $M$. If $S$ has non-exceptional Schur multiplier then the proof of Theorem~8.4 in \cite{Spaeth_AM_red} applies since the assumption that $p\geq 5$ is assumed there only to ensure that the Schur multiplier of $S$ is non-exceptional.

\subsection{Exceptional Sylow normalizers.}\label{n=ell= 2,3} 
In the next step we consider the cases where the description of  Sylow subgroups and their normalizers given by \cite[5.14]{Ma06} does not apply, i.e., that the normalizers of the Sylow subgroups are not related to the normalizers of Sylow $\Phi_d$-tori (see the above Section \ref{sec_blocks}) as in the general case.
This concerns essentially special cases with the primes $\ell = 2$ for PSL$_2$ and $\ell =3$ for PSL$_3$. 

 \begin{prop}\label{iAM2} 
Assume $S=\PSL_2(  q)$ is a simple group. Then $S$ satisfies the inductive AM-condition from \cite[7.2]{Spaeth_AM_red} with respect to any prime.
 
 \end{prop}

 \begin{proof} We check the inductive AM-condition with respect to the prime $\ell$.
 From the above, we may assume that the covering group of $S$ is the corresponding $G=\SL_2(q)$ and that $\ell\nmid q$. In this case note that any $\ell$-subgroup of $G$ for an odd $\ell$ is cyclic. We then appeal to \cite[1.1]{KoshitaniSpaeth2} to get our statement. When $\ell =2$, the Sylow $2$-subgroup of $G$ is generalized quaternion, and the non-cyclic 2-subgroups contain their own centralizers, so by Brauer's third main theorem (\cite[6.7]{Navarro}), only the principal $2$-block has non-cyclic defect, in fact maximal defect. Our claim is then proved in \cite[9.6]{Spaeth_Dade} as a by-product of the checking of the inductive condition for the Dade Conjecture in this special case.
 \end{proof} 

 \begin{prop}\label{iAM3}  Assume $S=\PSL_3(\ep q)$ is a simple group with $\ep q=4$ or $7$ mod 9. Then $S$ satisfies the inductive AM-condition for the prime $\ell =3$ with respect to any Sylow $3$-subgroup of $S$.
 \end{prop} 
 
 Let $D$ be a Sylow $3$-subgroup of $G=\SL_3(\ep q)$, which can be taken as a subgroup of monomial matrices with $|D|= 27$.

 \begin{lem}\label{3_block}  All elements of $\Irr_{3 '}(G)$ and all elements of $\Irr (\NNN_G(D))=\Irr _{3'}(\NNN_G(D))$ are in the principal $3$-block.
 \end{lem} 

\begin{proof} 
One has $\Cent_G(D)\leq D$, so any 3-block of $\Cent_G(D)$ is the principal block, and therefore any 3-block of $\NNN_G(D)$ is the principal block by Brauer third main theorem, see \cite[6.7]{Navarro}. On the other hand, any element of $\Irr_{3 '}(G)$ has to be in a $3$-block of maximal defect. Then it is the principal block again by Brauer third main theorem and what has been said about $\NNN_G(D)$.
\end{proof}

\renewcommand{\proofname}{Proof of Proposition \ref{iAM3}}
 \begin{proof}
 As explained in Sect. \ref{ExcMul}, we can assume that $S=\PSL_3(\ep q)$ has $G=\SL_3(\ep q)$ as covering group. The assumption on $\ep q$ implies that $q-\ep$ is divisible by 3 but not by 9, so we take $G <\wG\leq \GL_3(\ep q)$ with $\wG/G$ the subgroup of order 3. Then $\wG$ induces the action of $\GL_3(\ep q)$ on $G$ since the subgroup of order $q-\ep$ of $\FF_{q^2}^\times$ has a Sylow 3-subgroup of order 3. Let $E$ be the subgroup of $\Aut (G)$ generated by field automorphisms and transpose-inverse composed with conjugacy by the matrix of the transposition $(1,3)$ as in \cite[\S 3.1, 3.2]{M08a}.

We apply Theorem~\ref{thm2_7} with $D$ the Sylow $3$-subgroup of $G$ used above and $M=\NNN_ G(D)$.  The set $\Irr_{3'}(G)$ has six characters, the first three being unipotent - hence restrictions of unique unipotent characters of $\wG$ - and the three others are components of the restriction of some element of $\Irr (\wG )$ which is their common $\Ind^\wG_G$. All extend to their stabilizer in $\w G$ thanks to \cite[3.3]{M08a}. The set $\Irr_{3'}(M)$ follows the same pattern and one has a $(\wG E)_D = \w M E$-equivariant bijection $\Omega\colon \Irr_{3'}(G)\to \Irr_{3'}(M)$ constructed by Malle in \cite[3.3]{M08a}. 

One can now easily define a map $\w\Omega\colon\Irr (\wG\mid\Irr_{3'}(G))\to\Irr (\w M\mid\Irr_{3'}(M))$ between sets with 10 elements, equivariant for multiplication by linear characters of $\w G/G\cong \w M/M$ and $E$-action. Both sets are inside the principal block since it is the only 3-block covering the unipotent block (apply for instance \cite[9.6]{Navarro}), so we have the second block condition of Theorem~\ref{thm2_7}. The other block requirement is ensured by Brauer's third Main Theorem.
\end{proof}

\renewcommand{\proofname}{Proof}

{}
\subsection{End of proof of Theorem 1.1} We now conclude that any simple group of type $S=\PSL_n(\ep q)$ satisfies the inductive AM-condition for $\ell$-blocks with maximal defect for any prime $\ell$. We can assume $\ell\nmid q$ by Sect. \ref{DefCha} above, and that $\bG^F=\SL_n(\ep q)$ is the universal covering of $S$ by Sect. \ref{ExcMul}. Recall that we denote by $d\geq 1$ the order of $q$ mod $\ell$ when $\ell$ is odd, mod $4$ when $\ell = 2$. By \cite[5.14]{Ma06} and Sect. \ref{n=ell= 2,3} above, we may assume that the Sylow $\Phi_d$-torus $\bS $ of $\bG$ is such that $M:=\NNN_\bG (\bS )^F$ satisfies the group theoretic requirements of Theorem~\ref{thm2_7} with $\w M=\NNN_\wbG (\bS )^F$ and $D$ any Sylow $\ell$-subgroup of $M$. 

Using \cite[\S 6]{CabSpaeth2}, we have a bijection $\w\Omega\colon \Irr (\w G\mid\Irr_\lp (G))\to \Irr (\w M\mid\Irr_\lp (M))$ which satisfies the character theoretic requirements of Theorem~\ref{thm2_7}. There remains to check the equations on blocks given in \ref{thm2_7ii} and \ref{thm2_7_block}. The first one is implied by the above Proposition~\ref{prop_block}. The requirement of Theorem~\ref{thm2_7_block} is ensured by Proposition~\ref{wG_b}.

This completes our proof.
 
\def\cprime{$'$} \def\cprime{$'$} \def\cprime{$'$}

\end{document}